\documentclass[11pt,a4paper,reqno]{amsart}
\usepackage{amsmath}
\usepackage{amsfonts}
\usepackage{amssymb}
\usepackage{amscd}
\usepackage{url}
\usepackage{enumerate}
\usepackage[pdftex,bookmarks=true]{hyperref}

\newcommand\R{{\mathbf{R}}}
\newcommand\C{{\mathbf{C}}}

\renewcommand\P{{\mathbf{P}}}

\newcommand\eps{\varepsilon}

\newcommand\CC{{\mathcal C}}

\newcommand\ep{\varepsilon}

\newcommand\Sl{\mathbf{Sl}}

\newcommand\supp{\operatorname{supp}}

\theoremstyle{plain}
  \newtheorem{theorem}[subsection]{Theorem}
  \newtheorem{conjecture}[subsection]{Conjecture}
  \newtheorem{problem}[subsection]{Problem}

  \newtheorem{fact}[subsection]{Fact}
  \newtheorem{lemma}[subsection]{Lemma}

  \newtheorem{example}[subsection]{Example}

  \newtheorem{claim}[subsection]{Claim}

\theoremstyle{definition}
  \newtheorem{definition}[subsection]{Definition}

\title[Inhomogeneous random walks]{Anti-concentration of inhomogeneous random walks}

\author{Hoi H. Nguyen}
\address{Department of Mathematics, The Ohio State University, Columbus, OH 43210, USA}
\email{nguyen.1261@math.osu.edu}

\thanks{The author is partly supported by research grant DMS-1600782}

\begin{document}
\maketitle

\begin{abstract} We provide a characterization for
  anti-concentration of inhomogeneous random walks in 
  non-abelian groups. In application we extend the classical bounds by Erd\H{o}s-Littlewood-Offord and S\'ark\"ozy-Szemer\'edi to non-abelian settings.
 \end{abstract}

\section{Introduction}\label{section:introduction}

 Let $G=(G,\cdot)$ be an ambient  group which is not necessarily abelian. Let  $A_1, \dots, A_n$ be finite but not necessarily symmetric sets. Let $\mu_i$ be any probability distribution on $A_i$ such that

\begin{equation}\label{eqn:p_0}
\min_i \Big \{ \mu_i(a), a\in A_i \Big \}>p_0, 
\end{equation}

for some parameter $p_0>0$ which is allowed to depend on $n$ in some cases. 

We define the concentration probability of the random walk generated by $\mu_1,\dots,\mu_n$ to be 

$$\rho(\mu_1,\dots,\mu_n): = \|\mu_n* \dots * \mu_1\|_\infty = \max_{g\in G} \mu_n* \dots * \mu_1(g) .$$

Here the discrete convolution is defined as 

$$\mu*\nu(g) := \sum_{h\in \supp(\mu)} \mu(h)\nu(h^{-1}g).$$

Thus in contrast to the classical setting of random walks, our concern
here is on {\it inhomogeneous}  ones where the supports  $A_i$ of
$\mu_i$ can be totally different. 

In the abelian setting with $G=\C$ and with $\mu_i(a_i) =\mu_i(-a_i)=1/2$, the classical result of Erd\H{o}s \cite{E} and Littlewood-Offord \cite{LO} shows 

\begin{theorem}[forward Erd\H{o}s-Littlewood-Offord]\label{theorem:LO} Assume that $a_i$ are all non-zero complex numbers, then

\begin{equation}\label{eqn:ELO}
\rho(\mu_1,\dots,\mu_n) \le \frac{\binom{n}{\lfloor n/2 \rfloor}}{2^n}.
\end{equation}
\end{theorem}

This result was improved later by S\'ark\"ozy and Szemer\'edi \cite{SSz} (see also \cite{EM,Stan,Ng}) under an extra assumption.

\begin{theorem}[forward S\'ark\"ozy-Szemer\'edi]\label{theorem:SSz} Assume that $a_i$ are distinct complex numbers, then
\begin{equation}\label{eqn:SSz}
\rho(\mu_1,\dots,\mu_n) \le  (\sqrt{\frac{24}{\pi}}+o(1)) \frac{1}{n^{3/2}}.
\end{equation}
\end{theorem}

All of these results are optimal. We also refer the reader to the work by Hal\'asz \cite{Hal}, and to the survey \cite{NgV-survey} for further extensions and applications of these results.

\subsection{Non-abelian results} Although in the abelian setting the $a_i$ can be different, the ordering of the random steps does not matter. This is also the case for  classical random walks (in either abelian or non-abelian groups) of the form $\mu*\dots*\mu$. However, this pleasant property totally breaks down for inhomogeneous random walks in non-abelian groups, and this makes the analysis quite intractable. 

As far as we are concerned,  not much is known in the general non-abelian setting for inhomogeneous random walks. One related result we could find in the literature is from Varopolous' book \cite[Chapter VII 1.2.]{Var} where $G$ is a unimodular compactly generated group with polynomial volume growth of order $D$, and where the inhomogeneous random walk is generated by the $\mu_i$ of uniformly bounded density functions. It was shown in this case that the density function of $\mu_n*\dots*\mu_1$ is bounded from above by $n^{-D/2}$; we refer the reader to \cite{Var} for more details. 

Another result, which is directly relevant to our study, is a recent work by Pham and Vu \cite[Theorem 1.3]{TiepV}.

\begin{theorem}[forward Erd\H{o}s-Littlewood-Offord for matrices]\label{theorem:TiepVu}
Let $m,n,s$ be integers and let $a_i,1\le i\le n$ be elements of $GL_m(\C)$ with order at least $s$. Assume that $A_i=\{a_i,a_i^{-1}\}$ and $\mu_i(a_i)=\mu_i(a_i^{-1})=1/2$. Then

$$\rho(\mu_1,\dots,\mu_n) \le 141 \max\{\frac{1}{s}, \frac{1}{n^{1/2}}\}.$$ 
\end{theorem}

This bound is optimal up to the explicit multiplicative constant. 

One of the main goals of this note is to show the following analog of Theorem \ref{theorem:TiepVu} in asymptotic form.

 \begin{theorem}[forward Erd\H{o}s-Littlewood-Offord for general groups]\label{theorem:main:TiepVu} For any $\delta>0$, there exist
 $n_0$ and $0<\eps<1$ such that the following holds for $n\ge n_0$. Assume that the distributions $\mu_i$ in $G$ satisfy \eqref{eqn:p_0} with $p_0 \ge n^{-\eps^3}$ and such that each $A_i$ contains a pair of elements $a_i,a_i'$ with $a_i {a_i'}^{-1}$ being order of at least $s$, then 
  
  $$\rho(\mu_1,\dots,\mu_{n}) \le \max\{\frac{1}{s}, \frac{1}{n^{1/2-\delta}}\}.$$

 In particularly,  assume that $a_1,\dots,a_n$ are of order at least $s$ in $G$ and the supports $A_i$ contain $\{a_i,a_i^{-1}\}$, then the same conclusion holds.\end{theorem}
 
Next, motivated by S\'ark\"ozy-Szemer\'edi's result, one might also be interested in getting a non-trivial bound for $\rho(\mu_1,\dots,\mu_n)$ when the $\mu_i$ are essentially different. Our next result shows

\begin{theorem}[forward S\'ark\"ozy-Szemer\'edi for general groups]\label{theorem:main:SSz} For any $\delta>0$, there exist
 $n_0$ and $0<\eps<1$ such that the following holds for $n\ge n_0$. Assume that the distributions $\mu_i$ in $G$ satisfy \eqref{eqn:p_0} with $p_0 \ge n^{-\eps^3}$  such that each $A_i$ contains a pair of elements $a_i,a_i'$ so that  $a_i {a_i'}^{-1}, 1\le i\le n$, are all distinct. Then 
 
$$\rho(\mu_1,\dots,\mu_{n}) \le \frac{1}{n^{1-\delta}}.$$
   
In particularly, assume that $a_1,\dots,a_n$ are $n$ distinct elements of $G$ and the supports $A_i$ contain $\{id_G,a_i\}$, then the same conclusion holds.
\end{theorem}

In general the bound $n^{-1+o(1)}$ above is asymptotically sharp by the example that $a_i$ are elements of subgroups of $\Theta(n)$ elements. We also note that the conclusion fails in general if $A_i$ contains $\{a_i^{-1},a_i\}$ instead of $\{id_G,a_i\}$ because the $a_i$ might have order two.

\subsection{Method of proof} The way we prove Theorem \ref{theorem:main:TiepVu} and Theorem \ref{theorem:main:SSz} has its origin in \cite{TVstrong}. Notably, we will not be working directly with {\it forward} results as in Theorem \ref{theorem:LO},\ref{theorem:SSz},\ref{theorem:TiepVu} but with the {\it backward} ones. Roughly speaking, say if we want to prove Theorem \ref{theorem:main:TiepVu}, assume for contradiction that 

\begin{equation}\label{eqn:method:1}
\rho(\mu_1,\dots,\mu_n) \gg \max\{\frac{1}{s}, \frac{1}{n^{1/2-\delta}}\}.
\end{equation}

We then show that there exists a support $A_i$ of $\mu_i$ where $a_i{a_i'}^{-1}$ is of order smaller than $s$, this violates the assumption of the theorem. 

Similarly, to prove Theorem \ref{theorem:main:SSz} we assume for contradiction that 

\begin{equation}\label{eqn:method:2}
\rho(\mu_1,\dots,\mu_{n}) \gg \frac{1}{n^{1-\delta}}.
\end{equation}

Then there exists a set of size $o(n^{1/2})$ that contains most of the $a_i{a_i'}^{-1}$, which again contradicts our assumption.

The study of \eqref{eqn:method:1} and \eqref{eqn:method:2}, in its general framework, is called the {\it inverse Littlewood-Offord problem}. This was raised by Tao and Vu \cite{TVsurvey, TVinverse} about ten years ago.  

\begin{problem}\label{prob:LO:discrete} Characterize the sets $A_1,\dots,A_n$ when 

$$\rho(\mu_1,\dots,\mu_n)\ge n^{-O(1)}.$$
\end{problem}
  
We will devote the rest of this section to discuss this problem. To give an example of sets of large concentration probability, we first introduce some arithmetic structures.
  
\begin{definition}[progression]
Let $u_1,\dots,u_r$ be elements of $G$, and let $(N_1,\dots,N_r)$ be a vector of positive integers. Then the set of all products in the $u_i$ and their inverses in which each $u_i$ and its inverse appear at most $N_i$ times is called a {\it progression of rank $r$ and size lengths $N_1,\dots,N_r$}, and is denoted by $P(u_1,\dots,u_r;N_1,\dots,N_r)$ (or $P$ for short). 
\end{definition}

When $G$ is abelian, it is not hard to see that progressions grow very slow under addition in $G$. Thus if $A_1,\dots,A_n \subset P(u_1,\dots,u_r;N_1,\dots,N_r)$ with $\prod_i N_i =n^{O(1)}$ then $\rho(\mu_1,\dots,\mu_n) \ge n^{-O(1)}$. It was shown by Tao and Vu in \cite{TVinverse, TVstrong} (see also \cite{NgV} and \cite{Tao-LO}) that the converse is also true.

\begin{theorem}[inverse Erd\H{o}s-Littlewood-Offord]\label{theorem:ILO}
Let $G$ be a torsion-free abelian group. Let $A>0$ and $1>\eps>0$ be given constants, and let $m$ be any quantity between $n^\eps$ and  $n^{1-\eps}$.  Assume that $\mu_i(a_i)=\mu_i(-a_i)=1/2$ and  

$$\rho(\mu_1,\dots,\mu_n) \ge n^{-A}.$$ 

Then there exists a symmetric progression $P=P(u_1,\dots,u_r; N_1,\dots,N_r)$ of rank $r=O(1)$ and size $O(\rho^{-1}/m^{r/2})$ and there exist $n-m$ indices $i\in [1,n]$  such that 

$$A_{i} \subset P.$$
\end{theorem}

Our method develops a non-abelian counterpart of this result. We remark that the recent work by Tao \cite{Tao-LO}, among other things,  studies the distribution $\mu$ when $\rho(\mu,\dots,\mu) \ge n^{-A}$. The abelian inverse result, Theorem \ref{theorem:ILO} above, can be viewed as a very special case of the general framework of \cite{Tao-LO}, but the results there do not seem to directly cover our current setting of inhomogeneous random walks. 

Notice that for general $G$, one does not expect the condition $A_i\subset P$ to imply the largeness of $\rho(\mu_1,\dots,\mu_n)$. However, it would do if we know that the progressions $P$ are ``almost abelian".

\begin{definition}[nilprogression and coset nilprogression, \cite{BGT}] Suppose that $G$ is a group and  $r \ge 1,s\ge 0$ are integers. 

\begin{itemize} 
\item A {\it nilprogression} of rank $r$ and step $s$ is a progression $P(u_1,\dots,u_r;N_1,\dots, N_r)$ with the property that every iterated commutator of degree $s+1$ in the generators $u_1,\dots, u_r$ equals the identity $id_G$. 
\vskip .1in
\item A {\it coset nilprogression} of rank $r$ and step $s$ is a set of the form $\pi^{-1}(P )$, where $P$ is a nilprogression of rank $r$ and step $s$ in a quotient group $G_0/H$, where $H$ is a finite normal subgroup of a subgroup $G_0$ of $G$ and $\pi: G_0 \to G_0/H$ is the quotient map. 
\end{itemize}
\end{definition}

Thus coset nilprogressions can be written under the form $HP$, where $H$ is a finite subgroup which commutes as set with elements of the subgroup $\langle P \rangle $ generated by $P$.

 We next introduce a special type of nilprogression.

\begin{definition}[C-normal form, \cite{BGT}] Let $C\ge 1$. A nilprogression $P(u_1,\dots,u_r; N_1,\dots,N_r )$ is said to be in $C$-normal form if the following axioms are obeyed.

\begin{itemize}
\item (Upper triangular form) For every $i,j$ with $1\le i<j \le r$ and for all four choices of signs for the commutators 
$$[u_i^{\pm 1}, u_j^{\pm 1}] \in P(u_{j+1},\dots, u_r; \frac{CN_{j+1}}{N_iN_j}, \dots,\frac{CN_r}{N_iN_j}).$$

\vskip .1in

\item (Local properness) The expressions $u_1^{n_1} \cdots u_r^{n_r}$ are distinct as $n_1,\dots,n_r$ range over integers with 

$$|n_i|\le \frac{1}{C}N_i .$$

\vskip .1in

\item (Volume bound) One has 

$$\frac{1}{C} (2\lfloor N_1 \rfloor + 1) \dots (2\lfloor N_r \rfloor + 1) \le |P| \le C (2\lfloor N_1 \rfloor + 1) \dots (2\lfloor N_r \rfloor + 1).$$
\end{itemize}
 A coset nilprogression $\pi^{-1}(P )$ is said to be in $C$-normal form if the nilprogression $P$ is $C$-normal in the quotient group $G_0/H$.
\end{definition}

We also refer the reader to \cite{BT,Tointon} for several asymptotic equivalence between progressions and nilprogressions in nilpotent groups.  
A crucial property of coset nilprogressions in $C$-normal form is that their products grow polynomially slow (see  \cite[Proposition C.5]{BGT}),

\begin{equation}\label{eqn:growth:P}
|HP^n| = n^{O_{C,r}(1)}|H||P|.
\end{equation}

As such, similarly to the abelian case, coset nilprogressions are examples of sets of high concentration probability.

\begin{example}\label{example:discrete:general}
Assume that $HP$ is a nilprogression in $C$-normal form with rank $r$ and step $s$ of order $O(1)$, and with small cardinality $|HP|=n^{O(1)}$.

\begin{itemize}
\item Assume that $A_1,\dots,A_n \subset HP$, then by \eqref{eqn:growth:P} and by the pigeonhole principle,

$$\rho(\mu_1,\dots,\mu_n) \ge n^{-O(1)}.$$

\item More generally, assume that there is a finite set $X$ with $|X|=O(1)$ such that for each $1\le i\le n$ and each $a\in A_i$ there exists a permutation $\sigma_a \in Sym(X)$ such that for all $x\in X$,

$$a \in x HP (\sigma_a(x))^{-1}.$$

\noindent It is clear that in this case 

$$A_n \dots A_1\subset X HP \dots HP X^{-1} = XHP^nX^{-1}.$$ 

\noindent Hence $|A_n\dots A_1|=n^{O(1)}$, and so by the pigeon principle

$$\rho(\mu_1,\dots,\mu_n) \ge n^{-O(1)}.$$
\end{itemize}

\end{example}

By adapting the method of \cite{Tao-LO}, we will show the converse of the above.

\begin{theorem}\label{theorem:ILO:non-abelian} Let $G$ be a non-abelian group. Let $A>0$ an $1>\eps>0$ be given constants. Assume that the distributions $\mu_i$ in $G$ satisfy \eqref{eqn:p_0} with $p_0 \ge n^{-\eps^3}$ and such that

$$\rho=\rho(\mu_1,\dots,\mu_n) \ge n^{-A}.$$ 

Then there exists a coset nilprogression $HP$ with the following properties.

\begin{enumerate} 
\item $P$ is in $C$-normal form with $C=O(1)$ and with rank and step $r,s=O(1)$,

\vskip .1in

\item  $|HP|=O(\rho^{-1}),$ 

\vskip .1in
\item there is a finite set $X$ of cardinality $|X|=O(1)$, and consecutive indices $i_0,\dots, i_0+n'$ with $n'=n^{1-O(\eps)}$ such that the following holds: for each $a,a'\in A_i, i_0\le i\le i_0+n'$ there exists a permutation $\sigma \in Sym(X)$ such that for all $x\in X$,

$$a {a'}^{-1}\in x HP (\sigma (x))^{-1}.$$
\end{enumerate}
Here the implied constants depend on $\eps$ and $A$ but not on $G$. 
\end{theorem}

The bounds for $p_0$ and $n'$ above can be slightly improved but our final conclusion is not optimal,  we refer the reader to Conjecture \ref{conj:general} for a possible extension of this theorem. Although our characterization captures only $n'$ consecutive $\mu_i$ with some $n'=n^{1-O(\eps)}$ (in comparison to $n'=(1-o(1))n$ from Theorem \ref{theorem:ILO}), we can certainly run the argument at other segments; the obtained information is usually sufficient for asymptotic estimates.

Theorem \ref{theorem:ILO:non-abelian} heuristically supports the phenomenon that for the type of inhomogeneous random walks under consideration it is not {\it at all coincident} when the concentration probability is polynomially large at some sufficiently large step $n$. Indeed, {\it generic} inhomogeneous random walks should have {\it extremely small} concentration probability. To illustrate this point furthermore, allow us to give an example in the simple context of $\Sl_2(\R)$ in connection to the discrete Anderson-Bernoulli model in 1D. The result is by no mean important but we are not able to find similar statement in the literature. 

Consider the random walk generated by transfer matrices  $ g_{i}:= ( \begin{array}{cc}
E+\lambda \eps_i & -1 \\
1 & 0
\end{array} )
$ where $E,\lambda  \in \R, \lambda >0$ are given parameters, and where $\eps_i, 1\le i\le n$ are independent random variables with possibly different discrete distributions $\mu_i$ in $\R$ satisfying \eqref{eqn:p_0}. Assume furthermore that for any collection of $n^{1-\eps}$ consecutive distributions $\mu_{i_0},\dots, \mu_{i_0+n^{1-\eps}}$ there is a distribution $\mu_{i}$ whose support contains a symmetric pair $\{a_{i},-a_{i}\}$ with $a_{i}$ is greater than a given positive parameter $\gamma$. 

\begin{theorem}\label{theorem:sl2} Let be given $E, 0<\lambda, 0<\gamma  $ and $0<\eps<1$, the following holds for $\mu_1,\dots,\mu_n$ satisfying the above conditions with sufficiently large $n$ depending on $\lambda,\gamma$ and $\eps$
$$\sup_{g\in \Sl_2(\R)} \P(g_1\dots g_1=g)=n^{-\omega(1)}.$$
\end{theorem}
 
It is possible that the bound in Theorem \ref{theorem:sl2} is 
sub-exponential or even smaller, but we are unable to confirm this. Let us now discuss the proof of Theorem \ref{theorem:ILO:non-abelian}. To ease the presentation, we will decompose the proof into three parts.

\begin{enumerate} 
\item In the first step we will rely on the celebrated result by
  Breuillard, Green and Tao \cite{BGT} to obtain structures in the supports of large
  convolution sequences $\mu_{i_0+2l_0^\ast}*\dots *\mu_{i_0+l_0^\ast}$
 and $\mu_{i_0+l_0^\ast}*\dots *\mu_{i_0}$ for some $l_0^\ast=n^{1-o(1)}$ and
 $i_0=o(n)$. To arrive at the point of applying \cite{BGT}, we will use a simple dyadic argument and an asymmetric version of Balog-Szemer\'edi-Gowers theorem due to Tao \cite{T-product}.
\vskip .1in
\item In the second step, by following the mentioned work by Tao
  \cite{Tao-LO}, we obtain structures in the supports of smaller
  convolution sequences of type $\mu_{i_0+l_0^\ast}*\dots *\mu_{i_0+i}, 0\le i\le l_0$. The main focus of this step is on a semi-metric defined with respect to the structures obtained in Step 1.
\vskip .1in
\item In the last step, we improve upon Step 2 to obtain structures
  in the support of each individual $\mu_{i_0+i}$.
\end{enumerate}

As we can see, our proof of \ref{theorem:ILO:non-abelian} mainly relies on \cite{BGT} and \cite{Tao-LO}, so the
implicit constants of this result, and hence of Theorem \ref{theorem:main:TiepVu}, Theorem \ref{theorem:main:SSz} and Theorem \ref{theorem:sl2}, are ineffective.

{\bf Notation.} 
Throughout this paper, $n$ as an asymptotic parameter going to infinity.  We write $X = O_K(Y)$, $X \ll_K Y$, or $Y \gg_K X$ to denote the claim that $|X| \leq CY$ for some constant $C$ that depends on $K$. We also use $o(Y)$ to denote any quantity bounded in magnitude by $c(n) Y$ for some $c(n)$ that goes to zero as $n \to \infty$.  Again, the function $c(.)$ is permitted to depend on fixed quantities.

The rest of the note is organized as follows. The proof of Theorem \ref{theorem:ILO:non-abelian} is presented in Sections \ref{section:ILO:1}, \ref{section:ILO:2} and \ref{section:ILO:3}. Theorem \ref{theorem:main:TiepVu} will be shown in Section \ref{section:LO} by following the same ideas with some modifications. Finally, the proof of Theorem \ref{theorem:main:SSz} and Theorem \ref{theorem:sl2} will be presented in Section \ref{section:SSz} and Section \ref{section:sl2} respectively. 

\section{Proof of Theorem \ref{theorem:ILO:non-abelian}: first step}\label{section:ILO:1}

The main result of this section is Theorem \ref{theorem:BSzG:structure}. First of all, we introduce some elementary inequalities to be used.

\begin{claim}[Young's inequality]\label{claim:Young:discrete} Let $\mu$ and $\nu$ be probability measures with finite support in $G$. Then
\begin{itemize}
\item $$\|\mu\|_\infty \le \|\mu\|_2 \le \|\mu\|_\infty^{1/2};$$
\item $$\|\mu*\nu\|_\infty \le \|\mu\|_2 \|\nu\|_2;$$
\item $$\|\mu*\nu\|_2 \le \min\{\|\mu\|_2, \|\nu\|_2\}.$$
\end{itemize}
\end{claim}






Because of the second inequality, by passing to a subsequence of size at least $n/2$ when needed, instead of assuming $\|\mu_n*\dots*\mu_1\|_\infty\ge \rho$, we assume that 

\begin{equation}\label{eqn:passing:l2}
\|\mu_n*\dots*\mu_1\|_2^2 \ge \rho.
\end{equation}

For short, for $i<j$ we write 

$$\mu_{[i,j]}:= \mu_{j}* \dots * \mu_{i}.$$ 

\begin{claim}\label{claim:dyadic'} Let $0<\ep<1$ be given. There exist $i_0,l_0$ with $i_0+4l_0\le n$ and $l_0 \ge
n^{1-\ep/2}$ such that 

\begin{equation}\label{eqn:dyadic:l_2}
\|\mu_{[i_0,i_0+4l_0]}\|_2 \ge c \max_{i_0\le i\le i_0+7l_0/2} \|\mu_{[i,i+l_0/2]}\|_2,
\end{equation}
where $c$ is a sufficiently small constant depending on $\eps$.
\end{claim}

\begin{proof}(of Claim \ref{claim:dyadic'}) The proof is
  standard. Assume otherwise, then we can find a nested sequence $[1,n]  \supset [i_1,i_1+4l_1] \supset [i_2,4l_2] \supset \dots  \supset [i_k,i_k+4l_k] $ such that $l_{j+1} = l_j/8$ and that 
  
  $$\|\mu_{[i_j,i_j+4l_j]}\|_2 \le c  \|\mu_{[i_{j+1}, i_{j+1}+4l_{j+1}]}\|_2.$$
  
However, as $\|\mu_{[1,n]}\|_2 \ge n^{-O(1)}$ and $\|\mu_{[.]}\|_2 \le1$, the nested sequence above must have at most $k=O(\log_{1/c} n)$ terms. By definition 

$$l_k = \Omega(\frac{n}{8^{k}}) = \Omega(n^{1-\eps/2}).$$ 
\end{proof}

With $i_0,l_0$ from Claim \ref{claim:dyadic'} we have

$$\prod_{0\le j \le n^{\eps/2}/2}  \frac{ \|\mu_{[i_0+l_0- (j+1)n^{1-\eps}, i_0+2l_0]}\|_2}{\|\mu_{[i_0+l_0 - j n^{1-\eps}, i_0+2l_0]}\|_2} = \frac{\|\mu_{[i_0+l_0- (n^{\eps/2}/2+1)n^{1-\eps}, i_0+2l_0]}\|_2}{\|\mu_{[i_0+l_0, i_0+2l_0]}\|_2} \ge \frac{\|\mu_{[i_0,i_0+2l_0}\|_2}{\|\mu_{[i_0,i_0+l_0]}\|_2} \ge c.$$

Thus there exists $0\le j \le n^{\eps/2}/2$ such that 

\begin{equation}\label{eqn:dyadic:j}  
\frac{\|\mu_{[i_0+l_0- (j+1)n^{1-\eps}, i_0+2l_0}\|_2}{\|\mu_{[i_0+l_0 - j n^{1-\eps}, i_0+2l_0]}\|_2} \ge 1-n^{-\eps}.
\end{equation}

Set  $j_0:=i_0+l_0-j n^{1-\eps}$ and $l_0^\ast:= i_0+2l_0-j_0$. Then 

$$l_0^\ast \ge l_0 - n^{1-\eps/2}/2 \ge l_0/2.$$ 

Combine Claim \ref{claim:dyadic'} and \eqref{eqn:dyadic:j}, using the third monotonicity property from Claim \ref{claim:Young:discrete} we obtain

\vskip .2in

\begin{lemma}\label{lemma:dyadic} The exist $j_0,l_0^\ast$ with $l_0^\ast \ge n^{1-\eps}$ such that
\begin{equation}\label{eqn:discrete:1}
\|\mu_{[j_0-l_0^\ast,j_0+l_0^\ast]}\|_2 \ge c \max \Big\{\|\mu_{[j_0-l_0^\ast,j_0-1]}\|_2, \|\mu_{[j_0,j_0+l_0^\ast]}\|_2 \Big \}
\end{equation}
and
\begin{equation}\label{eqn:discrete:2}
|\mu_{[j_0-m,j_0+l_0^\ast]}\|_2 \ge (1-n^{-\eps}) \|\mu_{[j_0,j_0+l_0^\ast]}\|_2 \mbox{ for all } m\le n^{1-\eps}.  
\end{equation}
\end{lemma}

Note that although we vary $m$ in \eqref{eqn:discrete:2}, the inequality is clearly most meaningful at $m=n^{1-\eps}$. For the rest of this section, we will focus on \eqref{eqn:discrete:1}. For brevity, write 

$$\mu:= \mu_{[j_0,j_0+l_0^\ast]}, \mbox{ and } \nu:= \mu_{[j_0-l_0^\ast,j_0-1]}.$$ 

We can rewrite \eqref{eqn:discrete:1} to

\begin{equation}\label{eqn:assumption}
 c \max \Big \{\|\mu\|_2,\|\nu\|_2 \Big \} \le \|\mu*\nu\|_2 \le \min\Big \{\|\mu\|_2,\|\nu\|_2 \Big\}.
\end{equation}
  
To exploit this nice property, we will need an important notion of approximate group.

\begin{definition}\cite[Definition 1.2]{BGT}\label{def:approx} Let $K\ge 1$. A $K$-approximate group in a group $G$ is a multiplicative set $A$ with the following properties
\begin{itemize}
\item the set $A$ is symmetric: $id_G\in A$ and $a^{-1}\in A$ if $a\in A$;
\vskip .1in
\item there is a symmetric subset $X\subset A^3$ with $|X| \le K$ such that 
$$AA \subset XA.$$
\end{itemize}
\end{definition}

By using the asymmetric weighted Balog-Szemer\'edi-Gowers theorem we obtain the following analog of \cite[Proposition A.1]{BGGT}.

\begin{lemma}\label{lemma:BSzG}
Assume that $\mu$ and $\nu$ are probability measures such that 

$$\|\mu * \nu \|_2 \ge \frac{1}{K} \max\{\|\mu\|_2, \|\nu\|_2\}.$$

Then there is a $O(K^{O(1)})$-approximate subgroup $A$ of $G$ and $x_0,y_0\in G$ such that

$$|A| \ll K^{O(1)} (\max\{\|\mu\|_2, \|\nu\|_2\})^{-2}. $$

and 

$$\mu(x_0A), \nu(Ay_0) \gg K^{-O(1)}.$$
\end{lemma}

In application, as by \eqref{eqn:assumption} we will set 

$$K:=c^{-1}.$$

\begin{proof}(of Lemma \ref{lemma:BSzG})
We apply the machinery from \cite[Proposition A.1]{BGGT} and \cite[Theorem 4.6]{T-product}. Set 

$$\delta:= \frac{1}{100K^2} \mbox{ and } M= 10K.$$

Define

$$\mu':= \mu 1_{\mu \ge M \|\mu\|_2^2}, \mu'':= \mu 1_{\mu \le \delta \|\mu\|_2^2}, \mbox{ and } \tilde{\mu}:= \mu - \mu'-\mu''.$$

We note that 

$$\sum_{g\in \supp(\mu)} \mu'(g) \le \sum_{g\in \supp(\mu)} \mu'(g) \frac{\mu'(g)}{M \|\mu\|_2^2} \le \frac{1}{10K}.$$

Furthermore,

$$ \sum_{g\in \supp(\mu)} \mu''(g)^2 =  \sum_{g\in \supp(\mu)} \mu(g)^2 1_{\mu \le  \delta \|\mu\|_2^2} \le \delta \|\mu\|_2^2.$$

As such, by Young's inequality,

$$\|\mu'*\nu\|_2 \le \min \Big \{\|\mu'\|_2 \|\nu\|_1,\|\mu'\|_1 \|\nu\|_2\Big  \} \le \frac{1}{10K} \|\nu\|_2 \le \frac{1}{10} \|\mu*\nu\|_2,$$

and 

$$\|\mu'' *\nu\|_2 \le \min\Big \{\|\mu''\|_2 \|\nu\|_1,\|\mu''\|_1 \|\nu\|_2 \Big \} \le \delta^{1/2} \|\mu\|_2 \le \frac{1}{10} \|\mu* \nu\|_2.$$

Thus by the triangle inequality, $\|\tilde{\mu}\|_2$ and $\|\tilde{\mu}*\nu\|_2$ are comparable to $\|\mu*\nu\|_2$,

$$ \|\mu\|_2 \ge \|\tilde{\mu}\|_2 \ge \|\tilde{\mu}*\nu\|_2 \ge \frac{4}{5}\|\mu*\nu \|_2 \ge \frac{4c}{5} \|\mu\|_2.$$

By doing similarly with $\nu$, we obtain

\begin{equation}\label{eqn:tilde:mu,nu}
\|\tilde{\mu}*\tilde{\nu}\|_2\ge \frac{1}{2K}\max \{\|\tilde{\mu}\|_2, \|\tilde{\nu}\|_2\}. 
\end{equation}

Setting $B_1:= \supp(\tilde{\mu}), B_2:= \supp(\tilde{\nu})$. Then by definition of $\tilde{\mu}$ and $\tilde{\nu}$

$$|B_1|,|B_2| \asymp_K \|\mu\|_2^{-2} \mbox{ and } \mu(B_1), \nu(B_2) \asymp_K 1.$$

Also, by \eqref{eqn:tilde:mu,nu} 

$$E(B_1,B_2)\gg_K |B_1|^3,$$

where the implicit constants depend polynomially on $K$, and where $E(B_1,B_2)$ is the multiplicative energy, 

$$E(B_1,B_2):=\# \Big \{(b_1,b_1',b_2,b_2')\in (B_1^2\times B_2^2): b_1b_2=b_1'b_2')\Big\}.$$ 

By Tao's result on product set estimates for non-commutative groups  \cite[Theorem 5.2]{T-product}, there exist subsets $B_1'\subset B_1,B_2'\subset B_2$ with $|B_i'|\gg \frac{|B_i|}{K^{O(1)}}$ and such that 

$$|B_1' B_2'|\le K^{O(1)} |B_1|.$$

Also by  \cite[Theorem 4.6]{T-product}, there exists a $O(K^{O(1)})$-approximate group $A$ of size $O(K^{O(1)} |B_1|)$ and a finite set $Y$ of cardinality $O(K^{O(1)})$ such that 

$$B_1' \subset YA \mbox{ and } B_2'\subset AY.$$

Thus there exists $x_0\in Y$ and $y_0\in Y$ such that

$$|B_1' \cap x_0A|, |B_2'\cap A y_0| \ge |B_1|/O(K^{O(1)}) \ge  \|\mu\|_2^{-2}/O(K^{O(1)}).$$

This completes the proof of Lemma \ref{lemma:BSzG}.
\end{proof}

Our next ingredient is a simplified version \footnote{This result holds in more general setting where  $G$ can be local, see \cite{BGT}.} of the mentioned celebrated result by Breuillard, Green, and Tao. 

\begin{theorem}\cite[Theorem 2.10]{BGT}\label{theorem:BGT} Let $A$ be a finite $K$-approximate group in a global group $G$. Then $A^4$ contains a coset nilprogression $HP$ of rank and step $O_K(1)$ and $|P|\gg_K |A|$. Furthermore, $P$ can be taken to be in $O_K(1)$-normal form. 
\end{theorem}

Combine Lemma \ref{lemma:BSzG} with $K=c^{-1}$ and Theorem
\ref{theorem:BGT}, after a covering argument (as $A^4 \subset
X^3 A$ for approximate group $A$), we obtain the following.

\begin{theorem}\label{theorem:BSzG:structure}
Assume as in Lemma \ref{lemma:BSzG}, then there exists a coset nilprogression $HP$ (in $O_c(1)$-normal form) with  $|HP| \ll_c  \|\mu\|_2^{-2}$ and there exist $x_0,y_0\in G$ such that

$$\mu(x_0HP), \nu(HPy_0) \gg_c 1.$$
\end{theorem}

In particularly, Theorem \ref{theorem:BSzG:structure}  holds for $\mu$ and $\nu$ defined after Lemma \ref{lemma:dyadic}.

For later steps, it will be more convenient to pass to a sub nilprogression $Q$ of $P$ which is slightly more ``proper". Let $D\le 1/\eps$ be a constant to be chosen sufficiently large depending on other parameters (such as rank, step, $C$-normal form) of the structure $P$ obtained in Theorem \ref{theorem:BSzG:structure}. Consider the nillprogression

\begin{equation}\label{eqn:nil:Q}
Q:=P_{1/CD^2} = P(u_1,\dots,u_r; M_1,\dots,M_r) \mbox{ with } M_i=\frac{1}{CD^2}N_i.
\end{equation}

By definition, the following holds for $Q$

\begin{enumerate}[(i)]
\item for every $1\le i< j\le r$, 

$$[u_i^{\pm 1}, u_j^{\pm 1}] \in  P(u_{j+1},\dots, u_r; 
\frac{M_{j+1}}{D^2M_iM_j}, \dots,\frac{M_r}{D^2M_iM_j});$$

\vskip .1in
\item the expressions $u_1^{k_1} \cdots u_r^{k_r}$ are distinct for all $k_1,\dots,k_r$ with 

$$|k_i|\le D M_i;$$

\vskip .1in
\item 
$$|HQ| \asymp_c |HP|.$$
\end{enumerate}

We show that if $x\in HQ$ and $x^2\in HQ$ then $x$ is asymptotically an element of $HQ_{1/2}$. 

\begin{claim}\label{claim:proper} The following holds  with $D$ sufficiently large

\begin{enumerate}[(1)]
\item Assume that $x$ is an element of $Q$ where each $u_i$ and $u_i^{-1}$ appears with frequencies $n_i,n_i'$ respectively. Then $u$ can be written as $x=u_1^{m_1-m_1'}\dots u_n^{m_r-m_r'}$ with 

$$ \max\{|m_i-n_i|, |m_i'-n_i'|\} \le \frac{M_i}{D} \mbox{ for all } i.$$

\item Assume that $x=u_1^{n_1}\dots u_n^{n_r}\in Q$, then $x^2 =u_1^{m_1}\dots u_r^{m_2}$ with  

$$ |m_i -2n_i|\le \frac{M_i}{D}.$$

\item Assume that $x\in HQ$ and such that $x^2\in HQ$, then 

$$x\in HQ_{(1+\frac{1}{D})}.$$

\end{enumerate}

\end{claim}

We insert here a proof for completion.

\begin{proof}(of Claim \ref{claim:proper}) We will prove the first assertion, the second one follows similarly. Assume that in the representation of $x$ there are exactly $n_i^{(0)}=n_i$ and $n_i'^{(0)}=n_i'$ copies of $u_i$ and $u_i^{-1}$ respectively for $1\le i\le r$. We will move all copies of $u_1$ and $u_1^{-1}$ to the left. By (i), each step of replacing of $u_iu_1$ by $u_1u_i [u_i,u_1]$ would change the multiplicities $n_j^{(0)}$ of $u_j$ to $n_j^{(1)}$ where $ i+1\le j\le n$ and

$$|n_j^{(1)} - n_j^{(0)}| \le \frac{M_j}{D^2M_iM_1}.$$

Thus, after moving the first copy of $u_1$ all the way to the left after some $k_1\le 2(M_1+\dots+M_r)$ replacements, one has

\begin{equation}\label{eqn:nj}
|n_j^{(k_1)}- n_j^{(0)}| \le \frac{M_j}{D^2M_1}\sum_{i\le j-1} \frac{s_i^{(0)}}{M_i} \le  \frac{rM_j}{D^2 M_1},
\end{equation}

where we used the fact that $s_i^{(0)}$, the number of times $u_1$ meets $u_i$, is bounded by $s_i^{(0)} \le n_i^{(0)}\le M_i$. As a consequence, after $n_i$ steps of moving all $u_1$ to the left, one has \footnote{Strictly speaking, the bounds of $s_i^{(0)}$ from \eqref{eqn:nj} will increase after each round of moving a copy $u_1$ all the way to the left, but this change is negligible.}

\begin{equation}\label{eqn:nj'}
|n_j^{(k_1+\dots+k_{n_1})}- n_j^{(0)}| \le  n_j^{(0)} \cdot O(\frac{rM_j}{D^2 M_1}) =O(\frac{rM_j}{D^2}).
\end{equation}

Hence if we write  $x=u_1^{n_1-n_1'}y$, then in $y$ the $u_j, 2\le j\le r$ appears with total frequency $m_j$ where

$$n_j-O(rM_i/D^2) \le m_j \le n_j+O(rM_i/D^2).$$ 

We apply the collecting process again for $y$. The process terminates after $2r$ iterations, and at the end we obtain the desired bounds assuming $D$ to be large compared to $r$.

Now we  show the third claim for the case of nilprogression. We write $x$ in the form $u_1^{n_1}\dots u_n^{n_r}$ with $|n_i| \le (1+1/D)M_i$ as in (1). By the second assertion, 

$$x^2=u_1^{m_1}\dots u_r^{m_r}$$

with $|m_i - 2n_i| \le  M_i/D$.

However, as the elements $u_1^{k_1} \dots u_r^{k_r}$ are distinct for all $|k_1|\le DM_1,\dots,|k_r|\le DM_r$, and as $x^2\in Q$, by (1) we must have $|m_i|\le (1+1/D)M_i$, and so

$$2|n_i|-M_i/D \le (1+1/D)M_i.$$

Thus

$$|n_i| \le \frac{1}{2}(1+\frac{1}{D})M_i.$$

For the coset nilprogression case, note that if $x \in HQ$ and $x^2 \in HQ$ then $\pi(x)\in Q$ and $\pi^2(x)=\pi(x^2) \in Q$. We then argue as above for $\pi(x)$. 
\end{proof}

By using covering arguments, one sees that Theorem \ref{theorem:BSzG:structure} remains valid when $P$ is replaced by $Q$ (although with slightly worse constants). Without loss of generality we will assume our nilprogression $P$ to satisfy Claim \ref{claim:proper} from now on.

\section{Proof of Theorem \ref{theorem:ILO:non-abelian}: second step}\label{section:ILO:2}

We next continue our proof of Theorem \ref{theorem:ILO:non-abelian} by exploiting Theorem \ref{theorem:BSzG:structure} and equation \eqref{eqn:discrete:2} from Lemma \ref{lemma:dyadic}. Our main result of this section, Lemma \ref{lemma:HP,S:norm}, is obtained by following the approach of \cite{Tao-LO}. 

Set

$$n_0:=n^{\eps}.$$

For a given coset nilprogression $HP(x_1,\dots,x_r;N_1,\dots, N_r)$, and for $g \in \langle HP \rangle $, we define the norm of $g$ with respect to $HP$ to be

$$\|g\|_{HP}:=\inf \Big \{\lambda: g\in HP(x_1,\dots,x_r;\lambda N_1,\dots, \lambda N_r) \Big \}.$$

For short, we will denote  $HP(x_1,\dots,x_r;\lambda
N_1,\dots, \lambda N_r )$ by $HP_\lambda$. Note that in the special case that $\|g\|_{HP} <1/\max\{N_1,\dots,N_r\}$ then $g\in H$.

Next, for $g\in X \langle HP \rangle X^{-1}$ we also define the norm of $g$ with respect to $X$ and $HP$ as 

\begin{equation}\label{eqn:norm:S,HP}
\|g\|_{HP,X}:=\inf \Big \{ \lambda:  \exists \sigma\in Sym(X) \mbox{ so that } \forall x\in X,  g \in x HP_\lambda (\sigma(x))^{-1} \Big \}.
\end{equation}

Again, in the special case that $\|g\|_{HP,X} <1/\max\{N_1,\dots,N_r\}$ then there
exists $\sigma\in Sym(X)$ so that  for all $x\in X,  g \in x H(\sigma(x))^{-1}$.

Recall from the second property \eqref{eqn:discrete:2} of Lemma \ref{lemma:dyadic} that 

$$\|\mu*\eta_m\|_2 \ge (1-1/n_0)\|\mu\|_2$$ 

with  $\mu= \mu_{[j_0,j_0+l_0^\ast]}$ and $\eta_m= \mu_{[j_0-m, j_0-1]}$ for any $m\le n^{1-\eps}$. 

This  can be rewritten as

$$\int_G \int_G \|\mu*\delta_g  -\mu*\delta_h\|_2^2 d\eta_m(g) d \eta_m(h)  = 2 ( \|\mu\|_2^2 - \|\mu*\eta_m\|_2^2 )\le \frac{4}{n_0}\|\mu\|_2^2.$$



Motivated by this, we call a pair $(g,h)\in G^2$ in $\supp(\eta_m)\times \supp(\eta_m)$ {\it typical} if  

\begin{equation}\label{eqn:def:typical}
\|\mu*\delta_g  -\mu*\delta_h\|_2 \le \frac{1}{n_0^{1/2-\eps/2}}\|\mu\|_2.
\end{equation}

Note that $\eta_m$ has discrete support. Let $T_{\eta_m}$ denote the set of typical pairs.

\begin{claim}\label{claim:typpair:discrete} For $\eta_m$-asymptotically almost surely, any pair $(g,h)\in T_{\eta_m}$ is typical. More precisely,

$$\sum_{(g,h) \notin T_{\eta_m}} \eta_m(g) \eta_m(h) \le \frac{4}{n_0^\eps}.$$
 \end{claim}
 
\begin{proof}(of Claim \ref{claim:typpair:discrete})
By definition, 

$$\frac{1}{n_0^{1-\eps}}\|\mu\|_2^2 \sum_{(g,h) \notin T_{\eta_m}} \eta_m(g) \eta_m(h) \le \sum_{(g,h)\notin T_{\eta_m}} \|\mu*\delta_g  -\mu*\delta_h\|_2^2 \eta_m(g) \eta_m(h)  \le \frac{4}{n_0}\|\mu\|_2^2.$$

Thus 

$$\sum_{(g,h) \notin T_{\eta_m}} \eta_m(g) \eta_m(h) \le \frac{4}{n_0^\eps}.$$
\end{proof}

We next consider a typical pair $(g,h) \in T_{\eta_m}$. Notice that we can write $\mu*\delta_g(x) = \mu(xg^{-1})$ and $\mu*\delta_h(x) = \mu(xh^{-1})$. Thus, with $k=hg^{-1}$, by definition

\begin{equation}\label{eqn:typical'}
\sum_{x\in G} (\mu(x) - \mu (xk))^2 =\sum_{x\in G} (\mu(xg^{-1}) - \mu (xh^{-1}))^2 \le  \frac{1}{n_0^{1-\eps}} \sum_{x\in G} {\mu}^2(x).
\end{equation}

Thus it is natural to introduce the ``distance" with respect to $\mu$:

$$ d_{\mu}(g,h):= \sqrt{\frac{\sum_{x\in G} (\mu(xg^{-1}) - \mu (xh^{-1}))^2}{\sum_{x\in G} {\mu}^2(x)}}.$$

Thus $(g,h)$ is typical iff 

\begin{equation}\label{eqn:def:typical'}
d_\mu(g,h) \le \frac{1}{n_0^{1/2-\eps/2}}.
\end{equation}

Using definition, we can show the following elementary properties about $d_\mu$.

\begin{fact}\label{fact:dmu} For every $k$ we have $d_\mu(k, id_G)= d_\mu(k^{-1},id_G)$. Furthermore  $d_\mu$ is right-invariant, symmetric, and satisfies the triangle inequality.
\end{fact}

For the remaining part of this section we will continue to understand further properties of $d_\mu$ given the structure of $\supp(\mu)$ obtained from Theorem \ref{theorem:BSzG:structure}. As $\mu$ is fixed, allow us to drop the subscript $\mu$ in $d_\mu(.)$ for convenience.  We first show that the set of $k$ of small distance to $id_G$ can be covered efficiently.

\begin{claim}\label{claim:k:covering:discrete} For $\delta$ sufficiently small depending on $c$ (from Lemma  \ref{lemma:dyadic}),  there exists a collection of $O(c^{-O(1)})$-left translations of $HP^2$ which contains all $k$ with $d(k,id_G)\le \delta$.
\end{claim}

\begin{proof}(of Claim \ref{claim:k:covering:discrete}) Let $x_0HP$ be the coset nilprogression obtained from Theorem \ref{theorem:BSzG:structure}. By assumption $d(k^{-1}, id_G)=d(k, id_G)\le \delta$,

\begin{equation}\label{eqn:k}
\sum_{x\in x_0HP} (\mu(x) - \mu (xk^{-1}))^2\le \sum_{x\in G} (\mu(x) - \mu (xk^{-1}))^2 \le \delta^2 \|\mu\|_2^2.
\end{equation}

As $(\mu(x)-\mu(xk^{-1}))^2 \ge \frac{1}{2}\mu^2 (x) - \mu^2(xk^{-1})$, it follows that 

$$
\sum_{x\in x_0HP} {\mu}^2(xk^{-1}) \ge \frac{1}{2} \sum_{x\in x_0HP} {\mu}^2(x)  -\delta^2 \|\mu\|_2^2.
$$

Thus if we choose $\delta \le \delta_0$ with sufficiently small $\delta_0$ depending on $c$, 

\begin{equation}\label{eqn:k:1}
\sum_{x\in x_0HP} {\mu}^2(xk^{-1}) \ge K^{-O(1)} \|\mu\|_2^2.
\end{equation}

We now consider a maximal collection of {\it disjoint} left translations of the form 

\begin{equation}\label{eqn:collection}
\Big \{k_i HP, 0\le i\le N \Big \}, \mbox{ where } k_0=id_G \mbox{ and } d(k_i, id_G) \le \delta, i\ge 1.
\end{equation}

By disjointness (and as $P=P^{-1}$ and $HP=PH$),

$$x_0 HP k_i^{-1} \cap x_0 HP k_j^{-1} = \emptyset.$$

By \eqref{eqn:k:1} we must have 

$$N=K^{O(1)}.$$ 

By the maximality assumption, for any $k$ with $d(k,id_G)\le \delta$ there exists $k_i$ such that $kHP \cap k_i HP \neq \emptyset$, thus 

$$k\in k_iHP (HP)^{-1} \subset k_i HP^2.$$
\end{proof}

Let $C_0=N+1=O(K^{O(1)})=O(c^{-O(1)})$ be the constant obtained from the proof of Claim \ref{claim:k:covering:discrete}. We can always assume $C_0 \ge 2$. As we can always extend a maximal collection of disjoint translations of form \eqref{eqn:collection} with respect to $\delta_1$ (which plays the role of $\delta$ in Claim \ref{claim:k:covering:discrete}) to a maximal one with respect to $\delta_2 \ge \delta_1$, and because we have seen from the proof of Claim \ref{claim:k:covering:discrete} that each such maximal collection has at most $C_0$ members as long as $\delta\le \delta_0$, there exists an integer $l=O_K(1)$ such that 

\begin{equation}\label{eqn:CC}
\CC_{\delta_0/C_0^{l+1}} =\CC_{\delta_0/C_0^{l-1}},
\end{equation}
 
where $\CC_{\delta_0/C_0^{l+1}}$ and $\CC_{\delta_0/C_0^{l-1}}$ are such two maximal collections with respect to $\delta_1= \delta_0/C_0^{l+1}$ and $\delta_2=\delta_0/C_0^{l-1}$. Let $\{k_{l-1,1},\dots,k_{l-1,r}\}=\{k_{l+1,1},\dots,k_{l+1,r}\}$ be the representatives with respect to $\CC_{\delta_0/C_0^{l+1}}$ (and equivalently, with respect to $\CC_{\delta_0/C_0^{l-1}}$, and hence also with respect to $\CC_{\delta_0/C_0^l}$) where $r \le C_0$.

In the next step we define $T$ to be the collection of the left cosets $k_{l,i}\langle HP\rangle$. Because of our definition \eqref{eqn:collection} that every maximal collection contains $HP$, $T$ contains the coset $t_{id_G}=\langle HP \rangle$.  


One  can put a  ``distance" $d_T$ on the coset elements of $T$ as

$$d_T(x \langle HP \rangle ,y \langle HP \rangle):= \inf_{g\in G} \Big \{d(g,id_G):  gx \langle HP \rangle = y \langle HP \rangle \Big \}.$$

We remark that if $d_T(.)$ is well defined on the coset elements of $T$ then it does not depend on the representatives and it is symmetric. To show that it is well defined, for any vertex pair $(t,t')=(k_{l,i} \langle HP \rangle, k_{l,j} \langle HP \rangle)$ in $T$, because $d(k_{l,i},id_G)$ and $d(k_{l,j},id_G)$ are both finite,  $d_T(k_{l,i} \langle HP \rangle, id_G \langle HP \rangle)$ and $d_T(id_G \langle HP \rangle, k_{l,j} \langle HP \rangle)$ are finite, and so is  $d_T(k_{l,i} \langle HP \rangle, k_{l,j} \langle HP \rangle)$ by the triangle inequality with respect to $d(.)$. More precisely,

\begin{align}\label{eqn:tt'}
d_T(t,t') = d_T(k_{l,i} \langle HP \rangle, k_{l,j} \langle HP \rangle) &\le  d_T(k_{l,i} \langle HP \rangle, id_G \langle HP \rangle)+  d_T(id_G \langle HP \rangle, k_{l,j} \langle HP \rangle) \nonumber \\   
&\le d(k_{l,i},id_G) +  d(k_{l,j},id_G)\le 2\delta/C_0^{l+1} \le \delta/C_0^l.
\end{align}

Next we consider the weighted complete graph $G$ on $T$ with weights $w(f) = d_T(t,t')$ on any edge $f=(t,t') \in \binom{T}{2}$.
\begin{claim}\label{claim:tree}
There exists a spanning tree $F$ of $G$ with the following properties

\begin{enumerate}
\item for each pair $(t,t')\in \binom{T}{2}$, each weight of the edges on the tree path connecting $t$ to $t'$ is at most $d_T(t,t')$;
\vskip .1in 
\item one can also choose corresponding coset representatives $x_t$ for each $t\in T$ such that as long as $(t,t')$ is an edge of $F$
$$d_T(t,t')=d(x_t,x_{t'});$$
\item furthermore, for any $(t,t')\in \binom{T}{2}$

\begin{equation}\label{eqn:ts}
d_T(t,t') \asymp_K d(x_t,x_{t'}).
\end{equation}

\end{enumerate}

\end{claim}
The proof of this claim follows \cite[Lemma 3.2]{Tao-LO}, we present it here for the reader's convenience.

\begin{proof}(of Claim \ref{claim:tree}) We construct the tree and the coset representatives by a simple greedy algorithm starting from step 0 with $F_0=\{id_G\}$. Assume that at step $i$ we already obtain a subtree $F_i$ with the coset representative $x_t$ for each $t\in F_i$, we then find an edge connecting $F_i$ to $T\backslash{V(F_i)}$ of least weight, say $e=(t,t')$. It is clear that for any $t''\in F_i$ 

 \begin{equation}\label{eqn:min}
 d_T(t,t')\le d_T(t'',t'). 
 \end{equation}

Let $g$ be an element from $G$ such that $d(g,id_G) = d_T(t,t')$ by the definition of $d_T(.)$.  In the next step set $x_{t'}:= g x_t $ and $F_{i+1}: =F_i \cup \{e\}$, we continue the  process until the last vertex. 

The first claim then follows from \eqref{eqn:min} and the way $F$ was constructed. The second claim also follows because $x_t$ do not change along the construction process. For the third claim, first recall that $|V(T)| = O_K(1)$. Assume that $t_0=t,t_{1},\dots,t_{j-1}, t_j=t'$ is the $F$-path connecting $t$ to $t'$. By the triangle inequality 

\begin{align*}
d_T(t,t') &\le \sum_{t_0=t,t_{1},\dots,t_{j-1}, t_j=t', (t_{i},t_{i+1}) \in F} d_T(t_i,t_{i+1}) \\ 
&= \sum_{t_0=t,t_{1},\dots,t_{j-1}, t_j=t', (t_{i},t_{i+1}) \in F} d(x_{t_i},x_{t_{i+1}})  \le |V(T)| d(x_t,x_{t'}),
\end{align*}

where in the last estimate we used the first claim (1). For the other direction, again by (1) and by the triangle inequality

\begin{align*}
d_T(t,t') &\ge \frac{1}{j}\sum_{t_0=t,t_{1},\dots,t_{j-1}, t_j=t', (t_{i},t_{i+1}) \in F} d_T(t_i,t_{i+1}) \\ 
&= \frac{1}{j}\sum_{t_0=t,t_{1},\dots,t_{j-1}, t_j=t', (t_{i},t_{i+1}) \in F} d(x_{t_i},x_{t_{i+1}}) \ge \frac{1}{|V(T)|} d(x_t,x_{t'}).
\end{align*}

\end{proof}

Set 

$$X:=\{x_t: t\in T\}.$$

Then $|X|\le C_0 =O(K^{O(1)})$ and the cosets $x\langle HP \rangle, x\in X$ are disjoint. Furthermore, assume that $x$ comes from the coset $t=k_{l,i} \langle HP \rangle $, then 

\begin{equation}\label{eqn:dx}
d(x,id_G) \le |V(T)| d_T(t,t_{id_G}) \le |V(T)| d(k_{l,i},id_G) \le |V(T)|\delta_0/C_0^{l+1}\le \delta_0/C_0^{l},
\end{equation}

where we recall that $k_{l,i}$ is one of the representatives of $\CC_{\delta_0/C_0^{l+1}}$. 

By Claim \ref{claim:k:covering:discrete} and by \eqref{eqn:CC} we have 

$$x \in k_{l,i} HP^{2}.$$

In other words, 

$$k_{l,i} \in x HP^{2}.$$

We also notice that this holds for any representative $k_{l,i}$ where $x\in k_{l,i}\langle HP \rangle$. Thus, again by Claim \ref{claim:k:covering:discrete}  and \eqref{eqn:CC}

\begin{equation}\label{eqn:s:reverse}
\Big \{g: d(g,id_G)\le \delta/C_0^{l-1} \Big \} \subset \cup_{i=1}^r k_{l,i}HP^2 \subset \cup_{x\in X}^\ast x HP^4,
\end{equation}

where the disjointness comes from the mentioned fact that $x\langle HP \rangle, x\in X$, are disjoint.

We now establish the connection between $\|.\|_{HP,X}$ and $d(.)$.

\begin{lemma}\label{lemma:HP,S:norm} As long as $d(g,id_G) \le \delta_0/C_0^l$, we have
$$\|g\|_{HP,X} \ll d^{1-O(\eps)}(g,id_G).$$
\end{lemma}

\begin{proof}(of Lemma \ref{lemma:HP,S:norm}) Consider any $g$ with $d(g,id_G)\le \delta_0/C_0^l$. Then for any $x\in X$, 

$$d(gx,id_G) \le d(gx,x) + d(x,id_G) =d(g,id_G)+d(x,id_G) \le \delta_0/C_0^l+\delta_0/C_0^{l} \le \delta_0/C_0^{l-1},$$

where we used \eqref{eqn:dx}.

Thus, by \eqref{eqn:s:reverse}, $gx \in x' HP^4$ for some $x'\in X$. Write 

$$gx=x'h, \mbox{ for some } h\in HP^4.$$ 

Note that by the definition of $x,x'$

$$d(x,x') \le |V(T)| d(x\langle HP \rangle , x' \langle HP \rangle) \le |V(T)| d(g,id_G).$$ 

Thus

$$d(x',x'h)=d(x',gx) \le d(x',x)+d(x,gx)  \le (|V(T)|+1)d(g,id_G).$$

Again by right invariance and by the triangle inequality

$$d(id_G,h^q) = d(x',x' h^q) \le q  (|V(T)|+1) d(g,id_G).$$

Let $q$ be the largest power of 2 that is smaller than $\delta_0/C_0^l d(g,id_G)$

\begin{equation}\label{eqn:q,g}
 \frac{\delta_0}{2C_0^l d(g,id_G)}  \le q=2^k\le \frac{\delta_0}{C_0^l d(g,id_G)} .
\end{equation}

Thus $d(id_G,h^q) < \delta_0/C_0^{l-1}$, and so by \eqref{eqn:s:reverse} and by the fact that $h\in HP^4$

$$h^q \in  (\cup_{x\in X}^\ast x HP^4) \cap \langle HP \rangle.$$ 

Because $id_G\in X$ and the cosets $x\langle HP \rangle, x\in X$,  are all disjoint, we obtain

$$h^q \in HP^4.$$

By the properness of $HP$, after iterating the third conclusion of Claim \ref{claim:proper} $k$ times, we obtain that $h\in HP_{(1+1/D)^k/2^k} \subset HP_{1/q^{1-O(\eps)}}$ as $D$ was chosen to be larger than $1/\eps$. Hence,

\begin{equation}\label{eqn:h:HP}
\|h\|_{HP} = O(\frac{1}{q^{1-O(\eps)}}) \ll d^{1-O(\eps)}(g,id_G).
\end{equation}

Thus we have

$$\|g\|_{HP,X} \ll d^{1-O(\eps)}(g,id_G).$$

To complete the proof, we note that the map $x\to x'$ above depends on $g$ and it is one-to-one because the representatives $x$ come from different cosets of $\langle HP \rangle$.

\end{proof}



\section{Proof of Theorem \ref{theorem:ILO:non-abelian}: third step}\label{section:ILO:3}

We show the following form of Theorem \ref{theorem:ILO:non-abelian}.

\begin{theorem}[Structures for $a_m$'s]\label{theorem:ILO:non-abelian'} There exists a coset nilprogression $HP$ in $O(1)$-normal form of small rank and step with  $|HP| =O( \rho^{-1})$, and a finite set $X$ of cardinality $O(1)$ such that or all $1 \le m \le n^{1-\eps}$,

$$\|a_m {a_m'}^{-1}\|_{HP,X}^2 \le \frac{1}{n_0^{1-O(\eps)}}, \mbox{ for all } a_m,a_m' \in A_{j_0-m}.$$ 
\end{theorem}

\begin{proof}(of Theorem \ref{theorem:ILO:non-abelian'}) First observe that

$$\eta_m=\mu_{j_0}*\dots*\mu_{j_0-m} := \eta_{m-1}*\mu_{j_0-m}.$$

\begin{claim}\label{claim:typpair':discrete} For any $a,a' \in A_{j_0-m}$ there is a typical pair with respect to $\eta_m$ of the form $(ga,ha')$, where $(g,h)$ is also a typical pair with respect to $\eta_{m-1}$. 
\end{claim}
\begin{proof}(of Claim \ref{claim:typpair':discrete})
First of all, if we look at the typical pairs of $\eta_{m-1}$, then by Claim \ref{claim:typpair:discrete}

$$\sum_{(g,h) \notin T_{\eta_{m-1}}} \eta_{m-1}(g)\eta_{m-1}(h)\le \frac{4}{n_0^\eps}.$$ 

Let $T_{a,a'}$ be the collection of pairs of words $(g',h')$ in  $\supp(\eta_m) \times \supp(\eta_m)$ of the form $(ga,ha')$ where $(g,h)$ forms a typical pair with respect to $\eta_{m-1}$. Then by \eqref{eqn:p_0}

$$\sum_{(g',h')\in T_{a,a'}} \eta_m(g')\eta_m(h') \ge p_0^2 \sum_{(g,h)\in T_{\eta_{m-1}}} \eta_{m-1}(g)\eta_{m-1}(h)>p_0^2/2.$$ 

On the other hand, by Claim \ref{claim:typpair:discrete} applied to $\eta_m$

$$\sum_{(g',h') \notin T_{\eta_m}} \eta_m(g')\eta_m(h')\le \frac{4}{n_0^\eps}.$$ 

But as $p_0 \ge 1/n^{\eps^3} \ge (8/n_0^\eps)^{1/2}$, we have

 $$T_{a,a'} \cap T_{\nu} \neq \emptyset.$$
 
So there is a typical pair $(g,h)$ with respect to $\eta_m$ satisfying the conclusion.
\end{proof}

Let $HP$ be the coset nilprogression obtained from Theorem \ref{theorem:BSzG:structure}, for which by \eqref{eqn:passing:l2}

$$|HP|=O( \rho^{-1}).$$ 

For any $1\le m\le n^{1-\eps}$, and for any $a,a'\in A_{j_0-m}$ consider a $\nu_{m-1}$-typical pair $(g,h)$ so that $(ga,ha')$ is also a $\nu_m$-typical pair. By right invariance, 

$$d(ga{a'}^{-1},h)=d(ga, ha') \ll n_0^{-1/2+\eps/2}.$$ 

Thus 

$$ d(a' a^{-1} g^{-1}h, id_G)  \ll n_0^{-1/2+\eps/2}.$$

Furthermore, as $(g,h)$ is $\nu_{m-1}$-typical

$$d(g^{-1}h,id_G) \ll n_0^{-1/2+\eps/2}.$$

By the triangle inequality,

\begin{align}\label{eqn:dist:a}
d(a'a^{-1},id_G)  = d(a'a^{-1}g^{-1} h, g^{-1}h) &\le d(a'a^{-1}g^{-1} h, id_G) +  d(id_G, g^{-1}h)  \ll n_0^{-1/2+\eps/2}.
\end{align}

The proof of Theorem \ref{theorem:ILO:non-abelian'} is then complete by Lemma \ref{lemma:HP,S:norm}.
\end{proof}

We remark that the use of triangle inequality to obtain \eqref{eqn:dist:a} as above is rather wasteful. We suspect the following.

\begin{conjecture}\label{conj:general} Assume that $\mu_i$ are as in Theorem \ref{theorem:ILO:non-abelian} with $id_G\in A_i$ such that $\rho(\mu_1,\dots,\mu_n) \ge n^{-O(1)}$. Then there exist consecutive indices $i_0,\dots, i_0+n^{1-\eps}$ such that
$$\sum_{i_0\le i\le i_0+n^{1-\eps}} \sum_{a_i\in A_i}\|a_i\|_{HP,X}^2 \ll 1.$$
\end{conjecture}

This bound, if true, would be a non-abelian analog of \cite[Equation 7.9]{TV-JAMS} and it would directly yield the second conclusion of Theorem \ref{theorem:main:TiepVu}.

\section{The Erd\H{o}s-Littlewood-Offord bound in non-abelian groups}\label{section:LO}
To prove Theorem \ref{theorem:main:TiepVu} we will follow the proof of Theorem
\ref{theorem:ILO:non-abelian}. Assume for contradiction that for some sufficiently large constant 

$$\|\mu_n*\dots*\mu_1\|_\infty \ge C_0 \max\{\frac{1}{s}, \frac{1}{n^{1/2-\delta}}\}.$$

Without loss of generality (by passing to $n/2$ consecutive $\mu_i$, see also \eqref{eqn:passing:l2}) we can assume $\|\mu_n*\dots*\mu_1\|_2^2 \ge C_0 \max\{s^{-1}, n^{-1/2+\delta}\}$. We will choose $C_0$ to be larger than any other implied constants in the sequel.

Argue as in Section \ref{section:ILO:1}, by \eqref{eqn:dyadic:j},
there exists $0\le k \le n^{1-\eps}$ such that 

$$\frac{\| \mu_{[i_0+l_0-(j+1)n^{1-\eps}+k,
    i_0+2l_0]}\|_2}{\|\mu_{[i_0+l_0-(j+1)n^{1-\eps}+k+1,
    i_0+2l_0]}\|_2} \ge 1 -\frac{1}{n^{1-\eps}}.$$

Set $j_0=i_0+l_0-(j+1)n^{1-\eps}+k+1$ and $l_0^\ast =
i_0+2l_0-j_0$. We obtain the following analog of Lemma \ref{lemma:dyadic}.

\begin{lemma}\label{lemma:dyadic'} The exist $j_0,l_0^\ast$ with $l_0^\ast \ge n^{1-\eps}$ such that
\begin{equation}\label{eqn:discrete:1'}
\|\mu_{[j_0-l_0^\ast,j_0+l_0^\ast]}\|_2 \ge c \max \Big\{\|\mu_{[j_0-l_0^\ast,j_0-1]}\|_2, \|\mu_{[j_0,j_0+l_0^\ast]}\|_2 \Big \}
\end{equation}
and
\begin{equation}\label{eqn:discrete:2'}
|\mu_{[j_0-1,j_0+l_0^\ast]}\|_2 \ge (1-n^{1-\eps}) \|\mu_{[j_0,j_0+l_0^\ast]}\|_2.  
\end{equation}
\end{lemma}

Set $\mu:= \mu_{[j_0,j_0+l_0^\ast]}, \nu:=\mu_{[j_0-l_0^\ast, j_0-1]}$,
we follow Section \ref{section:ILO:1} to obtain Theorem
\ref{theorem:BSzG:structure} for $\mu,\nu$. 

In the next step, let 

$$n_0=n^{1-\eps} \mbox{ and } \eta = \mu_{j_0-1}.$$ 

Note that in contrast to Section \ref{section:ILO:2} and Section \ref{section:ILO:3}, our $n_0$ here is large and we will only focus on one special $\eta$ (instead of many $\eta_m$). By \eqref{eqn:discrete:2'} we have 

$$\|\mu*\eta \|_2 \ge (1-1/n_0)\|\mu\|_2.$$

By the argument of Section \ref{section:ILO:2}, especially by combining Claim \ref{claim:typpair:discrete} (for $\eta_m=\mu_{j_0-1}$), equation \eqref{eqn:def:typical'}  and Lemma \ref{lemma:HP,S:norm}, we obtain the following analog of
Theorem \ref{theorem:ILO:non-abelian'}.

\begin{theorem}
 There exists a coset nilprogression $HP$ in $O(1)$-normal form of
 small rank and step with  $|HP| =O( \frac{1}{C_0}\min\{s, n^{1/2-\delta}\})$, and a finite set
 $X$ of cardinality $O(1)$ and a distribution $\mu_{j_0-1}$ whose support contains a pair
 $\{a,a'\}$ such that $a{a'}^{-1}$ has order at least $s$ and

$$\|a{a'}^{-1} \|_{HP,X}^2 \le n_0^{-1+O(\eps)} < n^{-1+O(\eps)}.$$ 
\end{theorem}  

Now consider the bound $\|a{a'}^{-1}\|_{HP,X} \le n^{-1/2+C\eps}$ for some absolute constant $C$. If we choose $\eps$ so that $\delta>C\eps$, then 

$$|HP|=O(\min\{s, n^{-1/2+\delta}\}) = O(n^{1/2-\delta}) <n^{1/2-C\eps}.$$ 

Thus, the bound $\|a{a'}^{-1}\|_{HP,X} \le n^{-1/2+ C\eps}$ forces $p$ to be in $H$ for any representation of the form $xp\sigma(x)^{-1}$ of $a{a'}^{-1}$ with $p\in HP$. In other words, for all $x\in X$ 

$$a{a'}^{-1} \in x H \sigma(x)^{-1}.$$ 

Replace $x=\sigma(x)$ and iterate the relation $d$ times where $d$ is the order of $\sigma$ in $Sym(X)$. After multiplying the obtained identities, we have

$$(a{a'}^{-1})^{d}\in xHx^{-1}.$$ 

However, this would imply that the order $k$ of $a{a'}^{-1}$ is at most 

$$k \le d|H| =O(|HP|) = O(\frac{1}{C_0}s) <s,$$

where $C_0$ was chosen sufficiently large. This contradicts with our assumption that $a{a'}^{-1}$ must have order at least $s$.

\section{The S\'ark\"ozy-Szemer\'edi's bound in non-abelian groups}\label{section:SSz}
We prove Theorem \ref{theorem:main:SSz}. Assume otherwise, then again by passing to $n/2$ consecutive $\mu_i$  we can assume $\|\mu_n*\dots*\mu_1\|_2^2 \gg n^{-1+\delta}$. By Theorem  \ref{theorem:ILO:non-abelian'}, with $\eps=\delta/2$, there exists a coset nilprogression $HP$ with the following properties

\begin{enumerate} 
\item $P$ has  rank and step $r,s=O(1)$ and $|HP|=O(n^{1-\delta})$; 
\vskip .1in
\item There is a finite set $X$ of cardinality $|X|=O(1)$, and consecutive indices $i_0,\dots, i_0+n'$ with $n'=n^{1-\eps}$ such that 
$$\sup_{i_0\le i\le i_0 +n'} \|a_i{a_i'}^{-1} \|_{HP,X} < 1.$$
\end{enumerate}

More specifically, each element $a_i{a_i'}^{-1}, i_0\le i\le i_0+n^{1-\eps}$, can be written as $x h (x')^{-1}$ for some $x,x'\in X$ and $h\in HP$. However, this is impossible when $\eps=\delta/2$ because the $a_i{a_i'}^{-1}$ are distinct and there are only $|X|^2|HP| = O(n^{1-\delta})$ ways to choose for the values of $a_i{a_i'}^{-1}$ from the set $XHPX^{-1}$.

\section{Proof of Theorem \ref{theorem:sl2}}\label{section:sl2}

Assume otherwise that for some positive constant $A$

$$\rho=\sup_{g \in \Sl_2(\R)} \P(g_n \dots g_1=g) \ge n^{-A}.$$

By Theorem \ref{theorem:ILO:non-abelian}, there exists a nilprogression $HP$ with size $|HP|=O(n^A)$ and there exist a finite set $X$ of cardinality $|X|=O(1)$ and indices $i_0,\dots, i_0+n'$ with $n'=n^{1-O(\eps)}$ such that the following holds: for each $a\in \supp(\mu_i), i_0\le i\le i_0+n'$ there exists a permutation $\sigma_a \in Sym(X)$ such that for all $x\in X$,

$$\left( \begin{array}{cc}
E+\lambda a & -1 \\
1 & 0
\end{array}\right ) \in x HP (\sigma_a(x))^{-1}.$$

By our assumption, among these $n'$ consecutive $\mu_i$, there exists one whose support contains $a,-a$ with $a>\gamma$. We will be focusing on these two elements. For short, write

$$ g_{1}:= \left( \begin{array}{cc}
E+\lambda a & -1 \\
1 & 0
\end{array}\right )
\mbox{ and }  g_{2}:= \left( \begin{array}{cc}
E-\lambda a & -1 \\
1 & 0
\end{array}\right ).
$$

By definition, for any integer $k$ the ball $B_k(g_1,g_2)$ which consists of words of length at most  $k$ in $g_1^{\pm 1}, g_2^{\pm 1}$ has size

\begin{equation}\label{eqn:ball:poly}
|B_k(g_1,g_2)| \le |X HP^k X| = O(k^{O(1)} |HP|) = O(k^{O(1)} n^C),
\end{equation}

where we used \eqref{eqn:growth:P} in the estimate of $HP^k$.

On the other hand, $ (g_{1})^{-1}= ( \begin{array}{cc}
0 & 1 \\
-1 & E+\lambda a
\end{array} ) \mbox{ and } (g_{2})^{-1}= ( \begin{array}{cc}
0 & 1 \\
-1 & E-\lambda a
\end{array} ).
$ So

$$ h_1= g_{1} (g_{2})^{-1}= \left( \begin{array}{cc}
1 & 2\lambda a \\
0 & 1
\end{array}\right ) \mbox{ and }  h_2= (g_{1})^{-1} g_{2}= \left( \begin{array}{cc}
1 & 0 \\
2\lambda a & 1
\end{array}\right ).
$$

Choose $k_0= \lceil 1/ 2\lambda \rceil $ so that $2k_0 \lambda \ge 1$, and consider 

$$ h_1':=h_1^{k_0}= \left( \begin{array}{cc}
1 & 2k_0 \lambda a \\
0 & 1
\end{array}\right ) \mbox{ and }  h_2':=h_2^{k_0} = \left( \begin{array}{cc}
1 & 0 \\
2k_0 \lambda a & 1
\end{array}\right ).
$$

We next use the following lemma. 

\begin{lemma}\label{lemma:free}\cite{Br} If $\mu\in \R$ with $|\mu|\ge
  2$ then the group generated by the matrices 
$ \Big( \begin{array}{cc}
1 & \mu \\
0 & 1
\end{array} \Big) \mbox{ and }   \Big( \begin{array}{cc}
1 & 0 \\
\mu & 1
\end{array} \Big)$ is free.
\end{lemma}

Thus by Lemma \ref{lemma:free}, for any $k$ 

$$|B_k(g_1,g_2)| \ge |B_{k/2k_0}(h_1',h_2')| \ge 2^{k/2k_0}.$$

However this would contradict with the polynomial bound \eqref{eqn:ball:poly}.

\end{document}